\documentclass[12pt,reqno]{amsart}

\usepackage{ xy, graphicx, color, hyperref, tikz}
 \usepackage{amsmath}
\usepackage{amssymb }
\usepackage{amsfonts, hyperref}

 \usepackage{verbatim} 
 \usepackage{bbold}
\xyoption{all}

\title{Spinorial Representations of Orthogonal Groups}
\author{Jyotirmoy Ganguly}
\author{Rohit Joshi}

\newtheorem{theorem}{Theorem}
\newtheorem{corollary}{Corollary}
\newtheorem{lemma}{Lemma}

\newtheorem{prop}{Proposition}

\theoremstyle{definition}

\newtheorem{remark}{Remark}
\newcommand{\C}{\mathbb C}
\newcommand{\ztwo}{\mathbf Z/2\mathbf Z}

\newcommand{\nc}{\newcommand}

\nc{\thm}{\theorem}
\nc{\cor}{\corollary}
\nc{\mc}{\mathcal}
\nc{\mb}{\mathbb}
\nc{\mf}{\mathfrak}
\nc{\ul}{\underline}
\nc{\ol}{\overline}
\nc{\N}{\mb N}
\nc{\R}{\mb R}
\nc{\Z}{\mb Z}
\nc{\Q}{\mb Q}
\nc{\gm}{\gamma}

\nc{\dnu}{\frac{\partial}{\partial \nu}}

\nc{\Dnu}[1]{\frac{\partial^{#1}}{\partial \nu^{#1}}}

\nc{\dmo}{\DeclareMathOperator}

\dmo{\Ker}{Ker} \dmo{\val}{val} \dmo{\ord}{ord}

 \dmo{\pr}{pr}
\dmo{\I}{I}
\dmo{\II}{II}
\dmo{\odd}{odd}
\dmo{\sgn}{sgn}
\dmo{\id}{id}
\nc{\beq}{\begin{equation*}}
\nc{\eeq}{\end{equation*}}
\nc{\half}{\frac{1}{2}}
\dmo{\Mod}{mod}
\dmo{\dyn}{dyn}
\dmo{\simp}{sc}
\dmo{\im}{im}
\dmo{\Ad}{Ad}
\dmo{\Gal}{Gal}
\dmo{\Aut}{Aut}
\dmo{\Pin}{Pin}

\dmo{\core}{core}
\dmo{\res}{res}
\dmo{\lin}{lin}
 \dmo{\Spin}{Spin}
\dmo{\Sp}{Sp}
\dmo{\SL}{SL}
\dmo{\GL}{GL}
\dmo{\SO}{SO}
\dmo{\PGL}{PGL}
\dmo{\GSp}{GSp}
\dmo{\sd}{sd}
\dmo{\orth}{orth}
\dmo{\Span}{Span}
\dmo{\SU}{SU}
\nc{\la}{\lambda}

 \nc{\lip}{\langle}
 \nc{\rip}{\rangle}
\nc{\hpi}{\widehat{\pi}}

\dmo{\Perm}{Perm}
\dmo{\Res}{Res}
\dmo{\Ind}{Ind}
\dmo{\tr}{tr}
\dmo{\Sym}{Sym}
\dmo{\reg}{reg}
\dmo{\ch}{ch}
\dmo{\Hom}{Hom}
\dmo{\diag}{diag}
\dmo{\Or}{O}
\nc{\eps}{\varepsilon}

\date{\today}
 
\address{The Institute of Mathematical Sciences, Chennai-600113, India} \email{jyotirmoy.math@gmail.com}  
 
\address{BHASKARACHARYA PRATISHTHANA, Pune-411004, India} \email{rohitsj@students.iiserpune.ac.in}  
 
\begin{document}
\maketitle
\begin{abstract}
Let $G$ be a real compact Lie group, such that  $G=G^0\rtimes C_2$, with $G^0$ simple. Here $G^0$ is the connected component of $G$ containing the identity and $C_2$ is the cyclic group of order $2$. We give a criterion whether an orthogonal representation $\pi:G\to \Or(V)$ lifts to $\Pin(V)$ in terms of the highest weights of $\pi$. We also calculate the first and second Stiefel-Whitney classes of the representations of the Orthogonal groups.
\end{abstract}

\tableofcontents

\section{Introduction}\label{intro}

Let $G$ be a real compact Lie group such that $G = G^0 \rtimes C_2$, where $G^0$ is its connected component containing the identity and $C_2$ denotes the cyclic group of order $2$. We take $C_2=\{1,g_0\}$, where conjugation action of $g_0$ on $G^0$ gives a diagram automorphism of $G^0$. 
We in particular consider the groups $G$ with $G^{0}$ simple of type $A_n, D_n$ and $E_6$. These are of interest because other types do not admit a nontrivial diagram automorphism.

We call a real (resp. complex) representation $(\pi,V)$ of $G$ orthogonal if its image lies inside $\Or(V)$, the real (resp. complex) orthogonal group. We know that $\Pin(V)$ is a topological double cover of $\Or(V)$. Let $\rho:\Pin(V)\to \Or(V)$ denote the covering map. An orthogonal representation $\pi$ of $G$ is spinorial if there exists a Lie group homomorphism $\widehat{\pi}:G\to \Pin(V)$ such that the following diagram commutes:
\begin{center}
	\begin{tikzpicture}
	\node (A1) at (0,0) {$G$};
	\node (A2) at (3,0) {$\Or(V)$};
	\node (B1) at (3,3) {$\Pin(V)$};
	
	\draw[->](A1)to node [above]{$\pi$} (A2);
	\draw[->](B1) to node [right]{$\rho$} (A2);
	\draw[densely dotted,->](A1) to node [above]{$\widehat{\pi}$} (B1);
	\end{tikzpicture}
\end{center}
i.e. $\rho\circ\widehat{\pi}=\pi$. We write $\Or(n)$ (resp. $\SO(n)$) for $\Or(n,\R)$ (resp. $\SO(n,\R)$). For real representations we take the quadratic form $$Q(x_1,x_2, \ldots, x_n) = -\sum x_i^2,$$
and consider the corresponding real Pin group.

For any orthogonal complex representation $\pi$ of a compact group $G$ there exists a real representation $\pi_0$ such that $\pi\cong\pi_0\otimes \C$. For details we refer the reader to \cite[Chapter $2$, Section $6$]{BrokerDieck}. Note that the representation $\pi$ is spinorial if and only if $\pi_0$ is spinorial. 

The irreducible representations of $G= G^0 \rtimes C_2$ arise in the following way. Take an irreducible representation $(\pi^{\lambda}, V^{\lambda})$ of $G^0$, parametrized by the highest weight $\lambda$. Denote the highest weight of the representation $ \phi(x) =\pi^{\lambda}(g_0 x g_0^{-1})$ by $g_0 \cdot \lambda$. Consider the representation $\rho^{\lambda}=\mathrm{Ind}_{G^0}^G(\pi^{\lambda})$. 
There are two possibilities.
\begin{enumerate}
	\item [Type I:] The representation $\rho^{\lambda}$ is irreducible. In this case we have $g_0\cdot\lambda\neq \lambda$ and $$\rho^{\lambda}\mid_{G^0} = \pi^{\lambda} \oplus \pi^{g_0 \cdot \lambda}.$$
	\item [Type II:] The representation $\rho^{\lambda}$ is reducible and 
	$$\rho^{\lambda}=\pi^{\lambda,+}\oplus \pi^{\lambda,-},$$
	such that $\dim \pi^{\lambda,+}=\dim \pi^{\lambda,-}=\dim \pi^{\lambda}$.
	In this case we have $g_0\cdot\lambda= \lambda$ and $$ \pi^{\lambda,\pm}\mid_{G^0} = \pi^{\lambda}.$$
\end{enumerate}

In fact every irreducible representation of $G$ is either of Type I or Type II. From \cite{joshi} we obtain a criterion for spinoriality of reductive connected algebraic groups over a field of characteristic zero. The criterion appears as the first condition in Theorem \ref{spin1}. We write $\mf g$ for the Lie algebra of $G$. Let $\mf{g} $ be simple. 
In the following theorem $p(\underline{\nu})$ is a certain constant related to group $G^0$ and $\chi_{\lambda}(C)$ is the trace of the Casimir element for the representation of $G^0$ with highest weight $\lambda$. For details we refer Sections \ref{notpre} and \ref{JS} of this paper. 
\begin{theorem}\label{spin1}
	An orthogonal representation of $G$ of Type I is spinorial if and only if both the following conditions hold: 
	\begin{enumerate}
		\item 
		$p(\underline{\nu})\cdot (\dim V^{\lambda} )\cdot \left( \dfrac{\chi_{\lambda}(C)+ \chi_{g_0 \cdot\lambda}(C)}{\dim \mf{g}}\right)\equiv 0\pmod 2,$ \\
		
		\item 
		$\dim V^{\lambda}\equiv 0\,\,\text{or}\,\,3\pmod4$.
	\end{enumerate}
\end{theorem}

\begin{theorem}\label{spin2}
 An orthogonal representation $\pi^{\lambda, \pm}$ of $G$ of Type II is spinorial if and only if both the following conditions hold: 
	\begin{enumerate}
		\item 
		$$\dfrac{ p(\underline{\nu}) \cdot\dim V^{\lambda} \cdot   \chi_{\lambda}(C)}{\dim \mf{g}}\equiv 0\pmod 2.$$ 
		\item 
		$\dim V^{\lambda}-\chi_{\pi^{\lambda,\pm}}(g_0)\equiv 0\,\,\text{or}\,\,6\pmod8$.
	\end{enumerate}
\end{theorem}

We also provide a formula for $\chi_{\pi^{\lambda,\pm}}(g_0)$ in terms of the highest weight $\lambda$ of the representation of $G^0$ (see Theorem \ref{chi}).



.

An orthogonal representation $\pi$ of  $G$ is spinorial if and only if $w_2(\pi) + w_1(\pi)\cup w_1(\pi)=0 ,$ where $w_1 $and $w_2$ are first and second Stiefel-Whitney classes of $\pi$.  
Let $m$ denote the multiplicity of $-1$ as an eigenvalue of $\pi^{\lambda,\pm}(g_0)$. Then for representations of Type II we have $m=(\dim V^{\lambda}-\chi_{\pi}(g_0))/2$.  
We use these results to compute the second Stiefel-Whitney classes for representations $\pi=\pi^{\lambda, \pm}$ of $\Or(n)$ for $n\geq 4$:
%
$$
w_2(\pi_0)=\dfrac{2(n-1)}{n(2n-1)}\cdot \dim V^{\lambda}(\lambda,\lambda+2\delta)w_2(\gamma^n)+\dfrac{m(m-1)}{2}e_{\mathrm{cup}},$$
where $\mathrm{e_{cup}}$ is a certain cup product and $\gamma_n$ is the $n$-plane vector bundle over the infinite Grassmannian $G_n$. For details see Section \ref{sw}. Note that $G_n= BO(n)$ where $BO(n)$ denote the classifying space for $\Or(n)$. 
We also calculate $w_2(\pi_0)$ for the cases when $\pi=\rho^{\lambda}$ is an irreducible representation of $\Or(2n)$ and irreducible representations $\pi$ of $\Or(2n+1)$.

We also have a character formula to detect spinoriality of representations of orthogonal groups.
A representation $\pi$ of $\Or(n)$ is spinorial if and only if both of the following conditions hold:
	\begin{enumerate}
	\item $\chi_{ \pi}(I) -\chi_{ \pi}(d_1) \equiv 0 \text{ or } 6 \pmod 8$,
	\item $\chi_{ \pi}(I) -\chi_{ \pi}(d_2) \equiv 0 \text{ or } 6 \pmod 8$.
\end{enumerate} Here $\chi_{\pi}$ is the character of $\pi$ and  $d_1 = \diag(-1,1,1 \ldots, 1) $ and $d_2 = \diag(-1,-1,1,\ldots ,1)$, where $\diag(a_1, \ldots,a_n)$ is the diagonal matrix with entries $a_i$. In fact the same formulae detect the spinoriality of orthogonal representations of $\GL(n, \R)$

The paper is arranged as follows. Section \ref{notpre} reviews the basic definitions and notations. We give a criterion for spinoriality of semidirect product and establish a connection between spinoriality of a Lie group and its maximal compact subgroup in Section \ref{semidirect}. We present brief reviews of the papers \cite{joshi} and \cite{wendt} highlighting the important results in Sections \ref{JS} and \ref{wendt} respectively. In Section \ref{main} we give criteria for spinoriality of an orthogonal, irreducible representation of compact real Lie groups $G$ of the form $G^{0}\rtimes C_2$ in terms of their highest weights. In Section \ref{red} we solve the case of reducible representations. Section \ref{ortho} deals with the  particular case of Orthogonal groups. We calculate first and second Stiefel-Whitney classes of real representations of Orthogonal groups in Section \ref{sw} . Here we obtain expressions of first and second stiefel whitney class in terms of highest weights of the representations. In section \ref{chfor} we provide criteria to detect spinorial representations of orthogonal groups in terms of character values. Finally in Section \ref{examples} we work out some examples like $\Or(2), \Or(4)$ and $\Or(8)$. 

\bigskip

{\bf Acknowledgements:} The authors would like to thank Dr. Steven Spallone for helpful conversations. The first author of this paper was supported by a post doctoral fellowship from IMSc, Chennai. The second author of this paper was supported by a post doctoral fellowship from Bhaskaracharya Pratishthan, Pune.

\section{Notation and Preliminaries}\label{notpre}

\subsection{Compact Lie Groups}
Let $G$ be a real compact Lie group such that $G= G^0 \rtimes C_2$ such that $G^0$ is simple. Let $T$ be a maximal torus of $G^0$ with Lie algebra $LT$. Let $X^{*}(T)$ (resp. $X_{*}(T)$)
be the character (resp. co-character) lattice of $T$. Consider an irreducible orthogonal representation $(\pi^{\lambda},V^{\lambda})$ of $G^0$, where $\lambda$ denotes the highest weight. 
Let $\mathfrak{g}$ denote the Lie algebra of $G^0$. 
Consider a co-character $\nu\in X_{*}(T)$. 
Note that $\lambda\in X^{*}(T)$. 
We have $\chi_{\lambda}(C)=(\lambda,\lambda+2\delta)$, where $(,)$ is the Killing form for $\mathfrak{g}$ and $\delta$ is half the sum of positive roots for $\mathfrak{g}$, and $C$ is the Casimir element for $\mathfrak{g}$. Let $\pi_1(G)=X_{*}(T)/Q(T)$, where $Q(T)$ is the co-root lattice.
  (In Section \ref{JS} we review the usual pairing $\lip \alpha,\nu \rip$, fix norms on $LT$ and $LT^*$ associated to the Killing form.)
\subsection{Root Systems}
Let $R$ be the root system with respect to $T$. Let $\tau\in \Gamma$, where $\Gamma$ denotes the group of diagram automorphisms of the Dynkin diagram of $G^{0}$. In fact $\tau$ induces an outer automorphism of $G^0$. We in particular consider the groups $G=G^0\rtimes \langle g_0 \rangle$ such that the conjugation action of $g_0$ on $G^0$ gives an outer automorphism of $G^0$.   If $G^{0}$ is not simply connected we have $G^0=\hat{G^0}/Z_1$, where $\hat{G^0}$ denotes the
universal covering group of $G^0$, and $Z_1$ is some subgroup of the center of $\hat{G^0}$ .
Let $\Gamma_{Z_1}$ be the subgroup of $\Gamma$ which leaves $Z_1$ fixed. Define $\widetilde{G^0}=G^0\rtimes_{\phi} \Gamma_{Z_1}$, where $\phi : \Gamma \to \mathrm{Aut}(G^0)$ is a homomorphism (see \cite[section 2.1 ]{wendt} for details) . Take $\bar{G^0}\subset \widetilde{G^0}$ to be any sub extension of $G^0$. Write $S= (T^{\tau})^0$ for the connected component of group of fixed points of $\tau$ inside $T$ with Lie algebra $LS$. Define $S_0=T^{\tau}$, the sub torus fixed by $\tau$. Define 
$$R^{\tau}=\{\alpha\mid_{LS_0}: \alpha\in R\},$$ 
where $LS_0$ denotes the Lie algebra of $S_0$ and $R^1=R^{\tau\vee}$. Note that $R^{\tau}$ is also a root system. Here $R^{\tau \vee}= \{\frac{2 \alpha}{(\alpha, \alpha)} \mid \alpha \in R^{\tau}\}$, where the bilinear form $(,)$ is a suitable multiple of Killing form such that $(\alpha, \alpha)=2$ for a long root $\alpha$.
Let $e:\mathbb{C}\to \mathbb{C}$ denote the exponential map. For $\mu\in LT^*$ we have the map $e(\mu):LT\to \mathbb{C}$ given by $e(\mu)(h)\mapsto e^{\mu(h)}$ for $h\in LT$.


Write $\rho^{\tau} = \half \sum_{\beta \in R^{1+}} \beta $.

\section{Lifting Criteria for Semi-direct Products}\label{semidirect}

%
%
%
%
%
%
%
%
%
%
%
%
%
%

One can detect the spinoriality of a representation of a Lie group from the spinoriality of its restrictions to certain subgroups. We state the result as the following lemma. For a Lie group $G$ let $G^0$ denote its connected component containing identity.  

\begin{lemma}\label{prod}
	Let $G$ be a Lie group and $H$ be a subgroup of it such that $G=G^0\cdot H$. Then any orthogonal representation $\pi$ of $G$ is spinorial if and only if $\pi_1=\pi\mid_{G^0}$ and $\pi_2=\pi\mid_H$ are spinorial and the lifts of $\pi_1$ and $\pi_2$ agree on $G^0\cap H$. 
\end{lemma}	

\begin{proof}
	
	If $\pi$ is spinorial then $\pi_1$ and $\pi_2$ are spinorial. For the converse let the representations $\pi_1$ and $\pi_2$ be spinorial. We write $\hat{\pi_i}$ for the lift of $\pi_i$.  
	We have
	$$G=\bigsqcup_{\alpha \in I} G^0h_{\alpha},$$  where $h_{\alpha}$'s are the representatives of the cosets of $G^0$ in $G$ and $I$ is an indexing set.
	Now any element of $g\in G$ can be written as $g=a\cdot h_{\alpha}$, where $a\in G^0$.
	We define the lift of $\pi$ as 
	$$\hat{\pi}(g)=\hat{\pi}(a\cdot h_{\alpha})=\hat{\pi}_1(a)\hat{\pi}_2(h_{\alpha}).$$
	For a different coset representation if we have $g=a'\cdot h_{\beta}$, then $a'^{-1}a=h_{\beta}h_{\alpha}^{-1}\in G^0\cap H$. Since $\hat{\pi}_1$ and $\hat{\pi}_2$ agree on $G^0\cap H$, we obtain $\hat{\pi}_1(a'^{-1}a)=\hat{\pi}_2(h_{\beta}h_{\alpha}^{-1})$. This gives $\hat{\pi}_1(a)\hat{\pi}_2(h_{\alpha})=\hat{\pi}_1(a')\hat{\pi}_2(h_{\beta})$. Therefore the lift $\hat{\pi}$ is well-defined.
	
To prove that $\hat{\pi}$ is a lift it suffices to show that $\hat{\pi}$ is a homomorphism, i.e. for two elements $g_1, g_2\in G$, we require
\begin{equation}\label{hom}
\hat{\pi}(g_1g_2)=\hat{\pi}(g_1)\hat{\pi}(g_2).
\end{equation}
Let $g_1=a_1h_{\beta}$, $g_2=a_2h_{\gamma}$ and $h_{\beta} \cdot h_{\gamma}\in G^0h_{\alpha}$. Therefore $h_{\beta} \cdot h_{\gamma}\cdot h_{\alpha}^{-1}\in G^0$. We write
	\begin{align*}
	\hat{\pi}(g_1g_2)&=\hat{\pi}(a_1h_{\beta}a_2h_{\gamma}h_{\alpha}^{-1}h_{\alpha})\\
	&= \hat{\pi}((a_1h_{\beta}a_2h_{\beta}^{-1})(h_{\beta}h_{\gamma}h_{\alpha}^{-1})h_{\alpha})\\
	&=\hat{\pi}_1((a_1h_{\beta}a_2h_{\beta}^{-1})(h_{\beta}h_{\gamma}h_{\alpha}^{-1}))\hat{\pi}_2(h_{\alpha})\\
		&=\hat{\pi}_1(a_1h_{\beta}a_2h_{\beta}^{-1})\hat{\pi}_1(h_{\beta}h_{\gamma}h_{\alpha}^{-1})\hat{\pi}_2(h_{\alpha}).
	\end{align*}
	Note that since $G^0$ is normal in $G$, we have $a_1h_{\beta}a_2h_{\beta}^{-1}\in G^0$. We can rewrite the requirement mentioned in \eqref{hom} as
	$$\hat{\pi}_1(a_1h_{\beta}a_2h_{\beta}^{-1})\hat{\pi}_1(h_{\beta}h_{\gamma}h_{\alpha}^{-1})\hat{\pi}_2(h_{\alpha})=\hat{\pi}_1(a_1)\hat{\pi}_2(h_{\beta})\hat{\pi}_1(a_2)\hat{\pi}_2(h_{\gamma}).$$
Consider the element 
	$$x=(\hat{\pi}_1(a_1h_{\beta}a_2h_{\beta}^{-1})\hat{\pi}_1(h_{\beta}h_{\gamma}h_{\alpha}^{-1})\hat{\pi}_2(h_{\alpha}))^{-1}\hat{\pi}_1(a_1)\hat{\pi}_2(h_{\beta})\hat{\pi}_1(a_2)\hat{\pi}_2(h_{\gamma}).$$
	Taking the image of $x$ under the covering map $\rho$ we obtain 
	$$\rho(x)=({\pi}_1(a_1h_{\beta}a_2h_{\beta}^{-1}){\pi}_1(h_{\beta}h_{\gamma}h_{\alpha}^{-1}){\pi}_2(h_{\alpha}))^{-1}{\pi}_1(a_1)\pi_2(h_{\beta}){\pi}_1(a_2){\pi}_2(h_{\gamma}).$$
	Since both $\pi_1$ and $\pi_2$ are restrictions of the same representation we replace them by $\pi$ for the convenience of the computation. Thus we obtain 
	\begin{align*}
	\rho(x)&=({\pi}(a_1h_{\beta}a_2h_{\beta}^{-1}){\pi}(h_{\beta}h_{\gamma}h_{\alpha}^{-1}){\pi}(h_{\alpha}))^{-1}{\pi}(a_1)\pi_2(h_{\beta}){\pi}(a_2){\pi}(h_{\gamma})\\
	&=(\pi(a_1h_{\beta}a_2h_{\gamma}))^{-1}\pi(a_1h_{\beta}a_2h_{\gamma})\\
	&=1.
	\end{align*}
	Therefore we should have $x=\pm 1$. If we fix $h_{\beta}$ and $h_{\gamma}$ then $x$ becomes a continuous function on $G^0\times G^0$. Note that $h_{\alpha}$ depends only on $h_{\beta}$ and $h_{\gamma}$. Since $ G^0\times  G^0$ is connected $x$ takes a constant value on this domain. Taking $a_1=a_2=1$ we obtain
	\begin{align*}
	x &=(\hat{\pi}_1(h_{\beta}h_{\beta}^{-1})\hat{\pi}_1(h_{\beta}h_{\gamma}h_{\alpha}^{-1})\hat{\pi}_2(h_{\alpha}))^{-1}\hat{\pi}_1(1)\hat{\pi}_2(h_{\beta})\hat{\pi}_1(1)\hat{\pi}_2(h_{\gamma})\\
	&=(\hat{\pi}_2(h_{\beta}h_{\gamma}h_{\alpha}^{-1}h_{\alpha}))^{-1}\hat{\pi}_2(h_{\beta}h_{\gamma})\quad \text{as}\quad h_{\beta}h_{\gamma}h_{\alpha}^{-1}\in G^0\cap H   \\ 
	&=1
	\end{align*}
	So $x$ takes the value $1$ at $1\times 1$, so $x$ takes the value $1$ on $ G^0\times  G^0$. Since this is true for all $h_{\beta},h_{\gamma}$, we conclude that $x=1$. In other words the map $\hat{\pi}$ is a homomorphism.

Next we claim that the map $\hat{\pi}$ is continuous. Since $G^0$ is open in $G$, so is $ G^0h_{\alpha}$. We have 
$$\hat{\pi}\mid_{ G^0h_{\alpha}}(ah_{\alpha})=\hat{\pi}_1(a)\hat{\pi}_2(h_{\alpha}),\quad \text{for $ah_{\alpha}\in  G^0h_{\alpha}$}.$$
Note that the set $\{ G^0h_{\alpha}\mid\alpha\in I\}$ forms an open cover of $G$ and $\hat{\pi}\mid_{ G^0h_{\alpha}}$ is continuous for all $\alpha\in I$. Therefore $\hat{\pi}$ is continuous.
%
\end{proof}
Consider a group $G$ with the following conditions:
\begin{enumerate}
	\item $G = G_1 \rtimes_{\phi} G_2$,
	\item $G_1$ is a connected Lie group,
	\item $G_2$ is a discrete group.
\end{enumerate}

We prove the following theorem.
\begin{thm}\label{spinsemi}
	A representation $\phi$ of $G$ is spinorial if and only if $\phi\mid_{G_1}$ and $\phi\mid_{G_2}$ are spinorial.
\end{thm}

\begin{proof}[Proof of Theorem $1$]
	Since $G=G_1\rtimes_{\phi} G_2$ we obtain 
	\begin{enumerate}
		\item 
		$G^1=G^0$ and $G=G_1\cdot G_2$, 
		\item 
		$G_1\cap G_2=\{e\}$.
	\end{enumerate}	
	Therefore the corollary follows from Lemma \ref{prod}.
\end{proof}

\begin{lemma} \label{L}
	Let $G$,  $G'$,  $H$ be connected real Lie groups and $\phi :H \rightarrow G$ be a homomorphism. Let  $\alpha :G' \rightarrow G$ be a cover.    Then $\phi$ can be lifted to $G'$  if and only if the image of $\phi_*$ in $\pi_1(G)$ is contained in the image of $\alpha_*$.
\end{lemma}

\begin{proof}
It follows from the lifting theorem in algebraic topology, that there is a unique continuous topological lift $\psi$, which takes identity of $H$ to identity of $G'$.  We will prove that $\psi$ is a group homomorphism.	
	
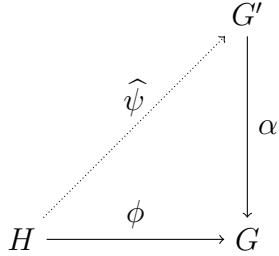
\begin{center}
		\begin{tikzpicture}
		\node (A1) at (0,0) {$H$};
		\node (A2) at (3,0) {$G$};
		\node (B1) at (3,3) {$G'$};
		
		\draw[->](A1)to node [above]{$\phi$} (A2);
		\draw[->](B1) to node [right]{$\alpha$} (A2);
		\draw[densely dotted,->](A1) to node [above]{$\widehat{\psi}$} (B1);
		\end{tikzpicture}
	\end{center}
 Let $*$ denote the multiplication in any group. We have $ \phi(g_1*g_2) = \phi(g_1)*\phi(g_2)$.  So we get $\alpha \circ \psi (g_1*g_2) = \alpha \psi (g_1) * \alpha \psi (g_2)$.  Hence $\alpha (\psi (g_1*g_2) \psi(g_1)^{-1} \psi(g_2)^{-1}) = 1$. The image of the map $p : H \times H \rightarrow G'$ given by $(g_1,  g_2) \rightarrow \psi (g_1*g_2) \psi(g_1)^{-1} \psi(g_2)^{-1}$ is connected, since $H$ is connected. The kernel of $\alpha$ is discrete, as it is a covering map.  Thus,  we get $(\psi (g_1*g_2) \psi(g_1)^{-1} \psi(g_2)^{-1}) = 1$.  Hence $\psi (g_1*g_2)= \psi(g_1) \psi(g_2)$.  Thus $\psi$ is, in fact, a group homomorphism.
	
\end{proof}

 In fact for a real, reductive Lie group $G$, the spinoriality of a representation of it can be detected by the spinoriality of its restriction to its maximal compact subgroup.

\begin{theorem}\label{thm2}
	Let $G$ be a reductive real Lie group such that $G^0$ has finite index in $G$. Let $K$ be a maximal compact subgroup of $G$. Then an orthogonal representation $\phi$ of $G$ is spinorial if and only if $\phi\mid_K$ is spinorial.
\end{theorem}

\begin{proof}
	One direction is obvious. Assume that $\phi\mid_K$ is spinorial. We denote the lift of $\phi\mid_K$ by $\hat{\phi}_K$. Note that  \cite[page $257$ Theorem $2.2$ part $3$]{HGSN} the inclusion map $i:K\to G$ is a homotopy equivalence, which means 
	\begin{enumerate}
		\item 
		We have $G=G^{\circ}\cdot K$.
		\item 
		The map $i_{*}:\pi_1(K^{\circ})\to \pi_1(G^{\circ})$ is an isomorphism.
	\end{enumerate}
	
Since $\pi\mid_K$ lifts, from Lemma \ref{L} we have $\phi_{*}(\pi_1(K^0))\subset \rho_{*}(\pi_1(\Spin(V)))$. Thus we have $\phi_{*}(\pi_1(G^0))\subset \rho_{*}(\pi_1(\Spin(V)))$ which says that $\phi\mid_{G^{\circ}}$ is spinorial. We denote the lift by $\hat{\phi}_0$. Since the lift $\hat{\phi}_0$ is unique we have $\hat{\phi}_0\mid_{K^0}=\hat{\phi}_K\mid_{K^0}$. Note that $K^0=G^0\cap K$.
Now the theorem follows from Lemma \ref{prod} by taking $H=K$.  
\end{proof}

\section{Review of Joshi-Spallone} \label{JS}

This paper gives critria for spinoriality of the representations of connected, reductive algebraic groups over fields of characteristic zero. We restrict our discussion to real, compact Lie groups. 
We continue with the notations as in section \ref{notpre}. Write $\langle,\rangle_T : X^{*}(T )\times X_{*}(T)\to \Z$ for the pairing
$$\langle \mu, \nu\rangle_T=n \Leftrightarrow \mu(\nu(t))=t^n$$
for $t \in \mathbb{R}^{\times}$ , and $\langle,\rangle_{LT}: LT^{*}\times LT$ for the natural pairing. Note that for
$\mu \in X^{*}(T)$ and $\nu\in X_{*}(T)$, we have
$$\langle d\mu, d\nu(1)\rangle_{LT} = \langle \mu, \nu\rangle_{T} .$$
So we may drop the subscripts and simply write $` \langle \mu, \nu\rangle'$.
Write $(,)$ for the Killing form of $\mathfrak{g}$ restricted to $LT$; it may be computed
by
$$(x, y) =\sum_{\alpha\in R}\alpha(x)\alpha(y),$$
for $x, y \in LT$. Also set $|x|^2 =(x, x)$. In particular, for $\nu \in X_{*}(T)$ we have
$$|\nu|^2 = \sum_{\alpha\in R} \langle \alpha, \nu \rangle^2 \in 2\Z.$$
The Killing form restricted further to $LT$ induces an isomorphism $\sigma :LT^*\cong LT$. We use the same notation $( , )$ to denote the inverse form on $LT^{*}$ defined for $\mu_1, \mu_2 \in LT$ by
$(\mu_1, \mu_2 ) = (\sigma(\mu_1),\sigma(\mu_2)).$

%

Pick co-characters $\nu_1, \ldots, \nu_r$ whose images generate $\pi_1(G)$, where $\pi_1(G)=X_{*}(T)/Q(T)$ i.e., co-character lattice modulo co-root lattice.
Consider the integer
\beq
p(\underline \nu)=p(\nu_1, \ldots, \nu_r)=\half \gcd \left( |\nu_1|^2, \ldots, |\nu_r|^2 \right).
\eeq
The following theorem follows from \cite[Theorem $1$]{joshi}
\begin{theorem}
The irreducible representation $\pi_{\lambda}$ of a connected reductive Lie group $G$ is spinorial if and only if
$$\dfrac{p(\underline\nu)\dim V^{\lambda}\chi_{\lambda}(C)}{\dim \mathfrak{g}}\equiv 0\pmod2,$$
\end{theorem}

\section{Review of Wendt}\label{wendt}

This paper \cite{wendt} gives Weyl character formula for character values of representations of real compact Lie groups $G$ with two connected components with $G^0$ of type $A_n, D_n$ and $E_6$. We continue with the notations as in section \ref{notpre}.
We write
$$\rho^{\tau}=\frac{1}{2}\sum_{\bar{\beta}\in R^1_{+}}\bar{\beta}.$$
We have $\delta^{\tau}:LS_0\to \C$, defined as $\delta^{\tau}(x)=e(\rho^{\tau}(x))\prod_{\bar{\beta}\in R_{+}^1}(1-e(\bar{\beta}(x)))$ for $x\in LS_0$.
For $\mu\in LS_0^{*}$, we have
$$A^{\tau}(\mu)=\sum_{w\in W^{\tau}}\epsilon(w)\cdot e(w\mu),$$
where $W^{\tau}$ is the Weyl group of the root system $R^{\tau}$ and $\epsilon$ is the sign character of the Weyl group $W^{\tau}$. Since $S_0$ is regular in $G$, we can choose a Weyl chamber $K\subset LT^{*}$ such that $K\cap LS_0^{*}$ is non-empty. Write $\bar{K}$ to denote the closure of $K$. Let $I$ denote the lattice $I=\ker(\exp)\cap LT$ and $I^{*}\subset LT^{*}$ be its dual.  For this paper we consider $G=\bar{G^0}=G^0 \rtimes C_2$, where $C_2=\lip g_0 \rip$.
Let $\pi^{\lambda}$ denote the irreducible representation of $G^0$ with highest weight $\lambda$. Write $\tau({\lambda})$ to denote the highest weight of the representation $\pi^{\lambda}(\tau\cdot g)$. Following Section \ref{intro} the representation $\rho^{\lambda}=\mathrm{Ind}_{G_0}^G\pi^{\lambda}$ is either of Type I or Type II. Let $\chi_{\lambda}$ denote the character of $\pi_{\lambda}$ and $\tilde{\chi}_{\lambda}$ denote the irreducible character of $G$. 
 We in particular take $\tau=C_{g_0}$, the conjugation action of $g_0$ on $G^0$.  We define a function $\tilde{\chi}_{\lambda}^{\tau}:LS_0\to \C$ as $\tilde{\chi}_{\lambda}^{\tau}(h)=\widetilde{\chi}_{\lambda}(g_0\cdot\exp(h))$. From \cite[Theorem $2.6$]{wendt} and \cite[Corollary $2.7$]{wendt} we obtain the following result.
\begin{theorem}\label{chwendt}
There exists an irreducible character $\widetilde{\chi}_{\lambda}$ of $\bar{G^0}$ for each $\lambda \in I^{*} \cap \bar{K}$. If $\lambda\notin LS_0^{*}$ then $\rho^{\lambda}$ is irreducible. In this case we have
$$\widetilde{\chi}_{\lambda}\mid_{G^0}=\chi_{\lambda}+\chi_{\tau(\lambda)},$$
and 
$$\widetilde{\chi}_{\lambda}\mid_{g_0G^0}=0,$$
		
For each $\lambda\in LS_0^{*}$, $\rho^{\lambda}$ splits into two irreducibles $\pi^{\pm}_{\lambda}$. In this case we have
$$\widetilde{\chi}_{\lambda}\mid_{G^0}=\chi_{\lambda},$$
and
$$\tilde{\chi}^{\tau}_{\lambda}(g_0\cdot \exp(h))=\widetilde{\chi}_{\lambda}^{\tau}(h)=\pm A^{\tau}(\lambda+\rho^{\tau})(h)/A^{\tau}(\rho^{\tau})(h),$$
where $h\in LS_0$.
\end{theorem}

From \cite[page $36$]{wendt} we have the following table showing relation between type of $R$ and type of $R^{\tau}$.

\begin{center}

\begin{table}[ht]
	\caption{ Table showing relation between type of $R$ and type of $R^{\tau}$.}
	\centering	
	
	\tabcolsep=.5cm
	\begin{tabular}{|p{.8cm}| p{.8cm}|p{.8cm}|p{2cm}|p{.8cm}|p{.8cm}|}
		\hline
		$R$ &$A_{2n-1}$ &  $A_{2n}$  & $D_n (n\geq 4)$ & $D_4$ &  $E_6$\\
		\hline
		$\mathrm{ord}(\tau)$  & $2$ & $2$    & $2$    &  $3$ & $2$\\
		\hline
		$R^{\tau}$  & $C_n$ & $BC_n$    & $B_{n-1}$    &  $G_2$ & $F_4$ \\
		
		\hline
	\end{tabular}
	\label{table}
\end{table}
\end{center}

 For the even orthogonal groups $\Or(2n)$, we take $\tau$ to be the automorphism of the Dynkin diagram of type $D_n$ switching two of the extremal nodes. It corresponds to the conjugation map by $$g_0=\left(\sum_{i=1}^{2n-2}e_{i,i}+e_{2n-1,2n}+e_{2n,2n-1}\right),$$
where $e_{i,j}$ denotes the elementary $2n\times 2n$ matrix with $1$ at the $(i,j)$-th position and $0$ everywhere else.
Let $R_{\theta}=\begin{pmatrix}
\cos\theta & \sin\theta\\
-\sin\theta & \cos\theta
\end{pmatrix}$. Let $T$ be the maximal torus 
$$T=\diag(R_{\theta_1},\ldots,R_{\theta_n}).$$ 
Thus 
$$S_0=\diag(R_{\theta_1},\ldots,R_{\theta_{n-1}}, I_2).$$

\section{Main Theorem}\label{main}

In this section we give the criteria for the spinoriality of the irreducible orthogonal representations of $G=G^0\rtimes C_2$. Let `$\sgn$' denote the non-trivial character of $C_2$. 
\begin{lemma}\label{lem1}
	Let $(\pi,V)$ be a representation of $C_2=\{\pm 1\}$, such that   $$\pi=\underbrace{\sgn\oplus\sgn\oplus\cdots\oplus\sgn}_{m\,\text{times}}\oplus\mathbb{1}\oplus\cdots\oplus\mathbb{1}.$$
	Then $\pi$ is spinorial if and only if $m\equiv 0\,\,\text{or}\,\,3\pmod4$, where $m = \dfrac{\chi_{ \pi}(1)- \chi_{ \pi}(-1)}{2}$.
\end{lemma}
The proof of Lemma \ref{lem1}  follows from \cite[Section $3.1$]{jyoti}.

\begin{thm}\label{chi}
	For the irreducible representation $(\pi^{\lambda,\pm},V)$ of $G$ we have 
	$$\chi_{\pi^{\lambda,\pm}}(g_0)=\pm\dfrac{\prod_{\alpha\in {R^{\tau}}^{+}}\langle \alpha^{\vee}, \lambda+\rho^{\tau}\rangle}{\prod_{\alpha\in {R^{\tau}}^{+}}\langle \alpha^{\vee}, \rho^{\tau}\rangle}.$$
\end{thm}

\begin{proof}

From Theorem \ref{chwendt} we obtain 
\begin{equation}
\chi^{\tau}_{\lambda}(h)=\pm \dfrac{A^{\tau}(\lambda+\rho^{\tau})}{\delta^{\tau}}(h),
\end{equation}
where $A^{\tau}(\mu)=\sum_{w\in W^{\tau}}\epsilon(w)e(w(\mu))$ and by \cite[Corollary $2.7$]{wendt} we have $\delta^{\tau}=A^{\tau}(\rho^{\tau})$. 
Note that
\begin{equation}\label{exp}
\chi_{\lambda}^{\tau}(h)=\widetilde{\chi}_{\lambda}( g_0\exp(h))=\chi_{ \pi^{\lambda , \pm}}(g_0\exp(h)),\,\,\text{for}\,\,h\in LS_0,
\end{equation}
For our case $\tau$ denotes the conjugation action by $g_0$. We have $\chi_{ \pi^{\lambda , \pm}}(g_0)=\chi^{\tau}_{\lambda}(0)$.
Let $\nu$ be a co-character of $LS_0$ such that $\nu(1)=h$. From Section \ref{JS} we obtain 
$$\langle w(\mu),xh\rangle=x\langle w(\mu),\nu(1)\rangle =x\langle w(\mu),\nu\rangle,$$
where $x\in \mathbb{R}$. We write $\langle w(\mu),\nu\rangle$ to denote $\langle w(\mu),\nu(1)\rangle$.
Therefore
\begin{equation}
A^{\tau}(\mu)(xh)=\sum_{w\in W^{\tau}}\epsilon(w)e(x\langle w(\mu),\nu\rangle).
\end{equation}

The character value $\tilde{\chi}_{\lambda}(g_0\exp(h))$ is a continuous function of $h$. Therefore the limit of the function exists at $h=0$. We take the limit along the line $xh\in LS_0$. This allows us to use Equation \eqref{exp} to calculate

$$\lim_{x\rightarrow 0} \dfrac{A^{\tau}(\lambda+\rho^{\tau})(xh)}{A^{\tau}(\rho^{\tau})(xh)}.$$
Note that $A^{\tau}(\lambda+\rho^{\tau})(0)=A^{\tau}(\rho^{\tau})(0)=0$. From L'Hospital's rule we have
$$\lim_{x\rightarrow 0} \dfrac{A^{\tau}(\lambda+\rho^{\tau})(xh)}{A^{\tau}(\rho^{\tau})(xh)}=\lim_{x\rightarrow 0} \dfrac{\frac{d^i}{dx^i}\left(A^{\tau}(\lambda+\rho^{\tau})(xh)\right)}{\frac{d^i}{dx^i}\left(A^{\tau}(\rho^{\tau})(xh)\right)},$$
where $i$ is the least natural number so that the numerator and denominator are non-zero. We calculate 
\begin{equation}
\frac{d^i}{dx^i}A^{\tau}(\lambda+\rho^{\tau})(xh)\mid_{x=0}=\sum_{w\in W^{\tau}}\epsilon(w)\langle w(\lambda + \rho^{\tau}),\nu\rangle^i.
\end{equation}
From \cite[Proposition $6$]{joshi} we obtain
\[
\sum_{w \in W}\epsilon(w)\langle w(\mu),\nu\rangle^i=
\begin{cases}
0,\,\,\text{if}\,\,1\leq i\leq N-1, \\
N!\dfrac{d_{\mu}d_{\nu}}{d_{\delta}},\,\,\text{if}\,\, i=N,
\end{cases}
\]
where $N$ denotes the number of positive roots of $R^{\tau}$ and $d_{\nu}=\prod_{\alpha\in {R^{\tau}}^{+}}\langle \alpha, \nu\rangle$ and $d_{\mu}=\prod_{\alpha\in {R^{\tau}}^{+}}\langle \alpha^{\vee}, \mu\rangle$. We calculate
\begin{align*}
\lim_{x\rightarrow 0} \dfrac{A^{\tau}(\lambda+\rho^{\tau})(xh)}{A^{\tau}(\rho^{\tau})(xh)}&=\dfrac{N!\frac{d_{\lambda+\rho^{\tau}}d_{\nu}}{d_{\delta}}}{N!\frac{d_{\rho^{\tau}}d_{\nu}}{d_{\delta}}}\\
&=\dfrac{d_{\lambda+\rho^{\tau}}}{d_{\rho^{\tau}}}\\
&=\dfrac{\prod_{\alpha\in {R^{\tau}}^{+}}\langle \alpha^{\vee}, \lambda+\rho^{\tau}\rangle}{\prod_{\alpha\in {R^{\tau}}^{+}}\langle \alpha^{\vee}, \rho^{\tau}\rangle}.
\end{align*}
Note that the expression $\prod_{\alpha\in {R^{\tau}}^{+}}\langle \alpha^{\vee}, \lambda+\rho^{\tau}\rangle$ is a polynomial in $\lambda$. Therefore it is non-zero for some $\lambda$.
This guarantees that the denominator should be non-zero. 
\end{proof}

\begin{remark}
Since the results in \cite{wendt} hold true for the groups $G$ with $G^0$ of type $D_n$, Theorem \ref{chi} also holds for the groups $\Or(2)$ and $\Or(4)$.
\end{remark}

\begin{prop}\label{eigendim}
	Let $\pi$ be an irreducible representation of $G$.	
	Let $m$ denote the multiplicity of $-1$ as an eigenvalue of $\pi(g_0)$. Then 
	
	\[
	m=
	\begin{cases}
	\dim V^{\lambda},\,\,\text{when}\,\,\pi\,\,\text{is of Type I},\\
	(\dim V^{\lambda} - \chi_{ \pi}(g_0))/2,\,\,\text{when}\,\,\pi \,\text{is of Type II}.
%
	\end{cases}
	\]
	
\end{prop}
\begin{proof}
	As $g_0$ is an involution we have	
	\begin{equation}\label{rhog0}
	\rho(g_0)\sim \begin{pmatrix}
	-I_m & 0 \\
	0 & I_l
	\end{pmatrix}.
	\end{equation}
	If $\rho^{\lambda}$ is irreducible then $l+m=2\dim V^{\lambda}$. From \cite[Theorem $2.6$]{wendt} it follows that $\chi_{\lambda}(g_0)=0$. This means $l-m=0$. So we deduce that $l=m=\dim V^{\lambda}$.
	
	For $\pi=\pi^{\lambda,\pm}$, $\chi_{\pi^{\lambda,\pm}}(g_0)=l-m$ and $l+m=\dim V^{\lambda}$. This gives the result. 
\end{proof}


\begin{proof}[Proof of Theorem \ref{spin1}]
Note that $\rho^{\lambda}\mid_{G^{0}}=\pi^{\lambda}\oplus \pi^{g_0\cdot \lambda}$, (see \ref{intro} for reference). From Theorem \ref{spinsemi} and Theorem \cite[Theorem $5$]{joshi}, we conclude that $\rho^{\lambda}\mid_{G^{0}}$ is spinorial if and only if

$$\dfrac{p(\underline\nu) \cdot(\dim V^\lambda )\cdot \left( \chi_{\lambda}(C) + \chi_{g_0 \cdot\lambda}(C ) \right)}{\dim \mf g}  \equiv 0\pmod 2,$$
Hence the first condition follows. To obtain the other condition use Lemma \ref{lem1} and Proposition \ref{eigendim}.
\end{proof}


\begin{proof}[Proof of Theorem \ref{spin2}]
The proof is similar to the previous one.
\end{proof}

\subsection{Case of $G^0\times C_2$}

Let $\pi_{\lambda}$ denote the  irreducible representation of $G^{0}$ and $\rho$ be an irreducible representation of $C_2$. We write $\sgn$ to denote the sign representation of $C_2$.
\begin{theorem}\label{directprod}
	An irreducible, orthogonal representation $\pi=\pi_{\lambda}\otimes \rho$ of $G$, for $n\in\N$, is spinorial if and only if both the following conditions hold:
	\begin{enumerate}
		\item 
		$$\dfrac{p(\underline \nu) \cdot(\dim V^\lambda ) \cdot \chi_{\lambda}(C)}{\dim \mf g}\equiv 0\pmod 2,$$  
		\item 
		$\dim V^{\lambda}\equiv 0\,\,\text{or}\,\,3\pmod4$ if  $\rho=\sgn$. Otherwise the first condition is sufficient.
	\end{enumerate}  
\end{theorem}

\begin{proof}
	We know that
	$$G=G^{0}\times C_2.$$
	Here we identify $C_2$ with $\{\pm 1\}\subset G$.
	By Theorem \ref{spinsemi}, $\pi$ is spinorial if and only if $\pi\mid_{G^{0}}$ and $\pi\mid_{C_2}$ are spinorial. We obtain the first condition from section \ref{JS}
	
	Note that $\pi(-1)= \pm I$ by Schur's Lemma. If it is $I$, then $\pi \mid_{C_2}$ lifts trivially. Otherwise the second condition comes due to Lemma \ref{lem1}.
\end{proof}

%
%

\section{Reducible Representations}\label{red}
As before we take $C_2=\langle g_0\rangle$.
Any orthogonal representation $(\pi,V)$ of a real compact group $G$ can be written as 
$$\pi=\oplus_i\rho_i\oplus_j(\phi_j\oplus\phi_j^{\vee}),$$
where $\rho_i$ is irreducible and orthogonal and $\phi_j$ is irreducible but not orthogonal. We have

\[
\phi_j=
\begin{cases}
\mathrm{Ind}_{G^0}^G\phi_{j0},\, \text{when $\phi_j$ is of Type I},\\
\psi_j^{\pm},\,\,\text{where}\,\,\mathrm{Ind}_{G^0}^G\phi_{j0}=\psi_j^{+}\oplus\psi_j^{-}, \text{when $\phi_j$ is of Type II},

\end{cases}
\]
where $\phi_{j0}$ is an irreducible representation of $G^0$.
For a representation $\pi$ of $G$, 
let $m_{\pi}$ denote the multiplicity of $-1$ as an eigenvalue of $\pi(g_0)$. Note that 
$$m_{\pi}=m_{\pi^{\vee}}.$$ 
Then we have
\[
m_{\phi_j\oplus\phi_j^{\vee}}=
\begin{cases}
2\dim \phi_{j0},\,\, \text{when}\,\, \phi_{j}\,\, \text{is of Type I}, \\
\dim \phi_{j0}-\chi_{\phi_{j}}(g_0),\,\, \text{when $\phi_j$ is of Type II}. 
\end{cases}
\]

\begin{thm}
Consider an orthogonal representation $\pi$ of $G$ of the form
$$\pi=\oplus_i\rho_i\oplus_j(\phi_j\oplus\phi_j^{\vee}).$$
Let $\lambda_i$ (resp. $\gamma_j$) denote the highest weight of $\rho_i$ (resp. $\phi_j$). Then $\pi$ 
is spinorial if and only if both the conditions hold:
\begin{enumerate}
\item 
\beq
q_\pi(\nu)= p(\underline \nu) \sum_i \frac{\dim \rho_i \cdot \chi_{\la_i}(C)}{\dim \mf g}\equiv 0\pmod 2,
\eeq

\item 
$m_{\pi}=\sum_im_{\rho_i}+\sum_jm_{(\phi_j\oplus\phi^{\vee})}\equiv 0\,\,\text{or}\,\,3\pmod 4.$
\end{enumerate}
\end{thm}

\section{Orthogonal Groups}\label{ortho}

\subsection{General Representations}
From \cite[Corollary $7.8$, page no. $297$]{BrokerDieck} we obtain that all the representations of the orthogonal group are orthogonal. We have 
\[
\Or(l)=
\begin{cases}
\SO(l)\times C_2,\quad\text{when $l=2n+1$},\\
\SO(l)\rtimes C_2,\quad\text{when $l=2n$}.
\end{cases}
\]
We call them odd and even orthogonal groups respectively. Take $C_2=\lip g_0 \rip $, where

\begin{equation}\label{g0}
g_0=
\begin{cases}
 -I_{2n+1},\,\,\text{for}\,\,\Or(2n+1),\\
\left(\sum_{i=1}^{2n-2}e_{i,i}+e_{2n-1,2n}+e_{2n,2n-1}\right),\,\,\text{for}\,\,\Or(2n).
\end{cases}
\end{equation}
Here $e_{i,j}$ denotes the elementary matrices.

\begin{corollary}
	An irreducible representation $\pi_{\lambda}\otimes \rho$ of $\Or(2n+1)$, for $n\in\N$, is spinorial if and only if both the conditions hold:
	\begin{enumerate}
		\item 
 $$\dfrac{(2n-1)}{n(2n+1)}\dim V^{\lambda}(\lambda,\lambda+2\delta)\equiv 0\pmod 2.$$
		\item 
		$\dim V^{\lambda}\equiv 0\,\,\text{or}\,\,3\pmod4$ if  $\rho=\sgn$. Otherwise the first condition is sufficient.
	\end{enumerate}  
\end{corollary}

\begin{proof}
	
	This follows from Theorem \ref{directprod} and \cite[Table $1$, Section $11$]{joshi}.
	\end{proof}

Again 
$$\Or(2n)=\SO(2n)\rtimes C_2.$$


Note that $D_2\cong A_1\oplus A_1$ and $D_3\cong A_3$. 
Let $\lambda=(\lambda_1, \lambda_2,\ldots, \lambda_n)$ be the highest weight of $\pi^{\lambda}$, then $\pi^{\lambda}$ is of type I when $\lambda_n \neq 0$. Otherwise it is of type II. For details we refer \cite[Section $7.5$, page $296$]{BrokerDieck}.
The following theorem gives the criteria for spinoriality for the representations of Type I.
\begin{corollary}
	The irreducible representation $\rho^{\lambda}$ of $\Or(2n), n\geq 3$, is spinorial if and only if both the following conditions hold: 
	\begin{enumerate}
		\item $$\dfrac{2 \cdot (n-1)}{n\cdot(2n-1)}\cdot\dim V^{\lambda}\left(\chi_{\lambda}(C) + \chi_{g_0\cdot\lambda}(C)\right) \equiv 0 \pmod 2
	$$
		\item 
		$\dim V^{\lambda}\equiv 0\,\,\text{or}\,\,3\pmod4$.
	\end{enumerate}
\end{corollary}

\begin{proof}
The result follows from Theorem \ref{spin1} and \cite[Remark $5$]{joshi}.
\end{proof}
 The next theorem gives lifting criteria for representations of Type II.
\begin{corollary}
	The representation $\pi_{\lambda}^{\pm}$ of $\Or(2n)$, for $n\geq 3$, is spinorial if and only if both the following conditions hold: 
	\begin{enumerate}
		\item 
		$$\dfrac{2 \cdot (n-1)}{n\cdot(2n-1)}\cdot\left(\dim V^{\lambda}\cdot\chi_{\lambda}(C)\right)\equiv 0\pmod 2,$$ 
		\item 
		$\dim V^{0}- \chi_{\pi}(g_0)\equiv 0\,\,\text{or}\,\,6\pmod8$, where $\pi = \pi^{\lambda, \pm}.$
	\end{enumerate}
\end{corollary}

\begin{proof}
This follows from Theorem \ref{spin2} and \cite[Remark $5$]{joshi}.
\end{proof}
Note that for $n\in \{1,2\}$ the group $\SO(n)$ is not simple. We work out those cases in Section \ref{examples}.
\subsection{Adjoint Representation}
The adjoint action of the Orthogonal group $\Or(n)$ on its Lie algebra $\mf{so}(n)$ preserves the Killing form $K(X,Y)=(n-2)Tr(XY)$ where $X,Y \in \mf{so}(n)$. 

\begin{theorem}
The Adjoint representation of the Orthogonal group $\Or(l)$ is spinorial if and only if $l\equiv 0\pmod 4$.
\end{theorem}

\begin{proof}
 From \cite[Corollary $4$]{joshi} we obtain $\mathrm{Ad}\mid_{\mathrm{SO}(l)}$ is spinorial if and only if $\delta\in X^{*}(T)$, where $\delta$ denotes half the sum of positive roots. For $l=2n+1$, the group $\SO(2n+1)$ is of type $B_n$. From \cite[Chapter $5$, Proposition $6.5$]{BrokerDieck}, we have 
$$\delta=\frac{1}{2}\sum_{i=1}^{n}(2n-2i+1)e_i\notin X^{*}(T).$$
Therefore the adjoint representations of the odd Orthogonal groups $\Or(2n+1)$ are aspinorial.

For $l=2n$, the group $\SO(2n)$ is of type $D_n$. In this case from \cite[Chapter $5$, Proposition $6.4$]{BrokerDieck}, we have
 
$$\delta=\sum_{i=1}^{n}(n-i)e_i\in X^{*}(T).$$
Therefore $\mathrm{Ad}\mid_{\mathrm{SO}(2n)}$ is spinorial. Following Theorem \ref{spinsemi} it remains to verify whether $\mathrm{Ad}\mid_{C_2}$ is spinorial.
For that we calculate 
$$\chi_{\mathrm{Ad}}(g_0)=\dfrac{\prod_{\alpha\in {R^{\tau}}^{+}}\langle \alpha^{\vee}, \lambda+\rho^{\tau}\rangle}{\prod_{\alpha\in {R^{\tau}}^{+}}\langle \alpha^{\vee}, \rho^{\tau}\rangle}.$$
From  Table \ref{table} we obtain $R^{\tau}$ is of type $B_{n-1}$. Therefore one calculates
\begin{align*}
\rho^{\tau}&=\frac{1}{2}\sum_{\alpha\in R^{\tau}}\frac{2\alpha}{(\alpha,\alpha)}\\
&=\frac{1}{2}\sum_{i<j}\{(e_i+e_j)+(e_i-e_j)\}+\frac{1}{2}\sum_{j=1}^{n-1}2e_j\\
&=\sum_{i=1}^{n-1}(n-i)e_i.
\end{align*}
From the given root system for the groups of type $D_n$ in \cite[Chapter $5$, Proposition $6.4$]{BrokerDieck}, we obtain the highest weight as $\lambda= e_1+e_2$. Therefore $$\lambda+\rho^{\tau}=ne_1+(n-1)e_2+\sum_{i=3}^{n-1}(n-i)e_i.$$ 
One calculates
\begin{align*}
\chi_{\mathrm{Ad}}(g_0)&=\dfrac{\prod_{\alpha\in {R^{\tau}}^{+}}\langle \alpha^{\vee}, \lambda+\rho^{\tau}\rangle}{\prod_{\alpha\in {R^{\tau}}^{+}}\langle \alpha^{\vee}, \rho^{\tau}\rangle}\\
&=\dfrac{\prod_{\alpha\in {R^{\tau}}^{+}}\langle \frac{2\alpha}{(\alpha,\alpha)}, \lambda+\rho^{\tau}\rangle}{\prod_{\alpha\in {R^{\tau}}^{+}}\langle \frac{2\alpha}{(\alpha,\alpha)}, \rho^{\tau}\rangle}\\
&=\dfrac{\prod_{\alpha\in {R^{\tau}}^{+}}\langle \alpha, \lambda+\rho^{\tau}\rangle}{\prod_{\alpha\in {R^{\tau}}^{+}}\langle \alpha, \rho^{\tau}\rangle}\\
\end{align*}
Note that $\lambda+\rho^{\tau}$ and $\rho^{\tau}$ differ only by first two terms. Therefore the terms in both the numerator and denominator containing the elements $e_1$ and $e_2$ will survive. The positive roots for groups of type $B_{n-1}$ are $e_i\pm e_j$ for $i<j$, and $e_j$ for $1\leq j\leq n-1$. Consider the set 
$$S=\{e_1\pm e_j\mid j>1\}\cup \{e_2\pm e_j\mid j>2\}\cup \{e_1,e_2\}.$$
We have
\begin{align*}
\chi_{\mathrm{Ad}}(g_0) &=\dfrac{\prod_{\alpha\in S}\langle \alpha, \lambda+\rho^{\tau}\rangle}{\prod_{\alpha\in S}\langle \alpha, \rho^{\tau}\rangle}\\
&=\dfrac{(2n-4)!/2\cdot (2n-3)!\cdot (2n-1)}{(2n-4)!\cdot (2n-5)!\cdot (2n-3)}\\
&=2n^2-5n+2.
\end{align*}
Following Proposition \ref{eigendim} we obtain the multiplicity of $-1$ as an eigenvalue of $\mathrm{Ad}(g_0)$ as 
\begin{align*}
\frac{1}{2}\cdot(\dim(\mathrm{Ad})-\chi_{\mathrm{Ad}}(g_0))&=\frac{1}{2}\cdot\left(\frac{(2n)^2-2n}{2}-(2n^2-5n+2)\right)\\
&=2n-1.
\end{align*}
Since $2n-1$ is odd, from Lemma \ref{lem1} we conclude that $\mathrm{Ad}\mid_{C_2}$ is spinorial if and only if $2n-1\equiv 3\pmod 4$. Equivalently we require $2n\equiv 0\pmod 4$.   
\end{proof}


\begin{remark}
In general the Adjoint representation of $G=G^0\rtimes C_2$ is of type II. This is because $\mathrm{Ad} \mid_{G^0}$ remains irreducible.
\end{remark}

\section{Stiefel-Whitney Classes for Representations of Orthogonal Groups}\label{sw}
For a brief introduction on Stifel-Whitney classes of representations of Lie groups we refer the reader to \cite[Section $2.6$, page no. $50$]{benson}.
We first calculate the second Stiefel-Whitney class for a representation of $C_2$. 
\begin{lemma}\label{lem2}
Let $(\pi,V)$ be a representation of $C_2$ and  $$\pi=\underbrace{\sgn\oplus\sgn\oplus\cdots\oplus\sgn}_{m\,\text{times}}\oplus\mathbb{1}\oplus\cdots\oplus\mathbb{1}$$
then 
$$w_2(\pi)=\frac{m(m-1)}{2}\cdot w_1(\sgn)\cup w_1(\sgn).$$
\end{lemma}

\begin{proof}
We have
$$w(\pi)=(1+w_1(\sgn))^m.$$
Therefore $w_2(\pi)=\binom{m}{2}\cdot w_1(\sgn)\cup w_1(\sgn)$.
\end{proof}
We write $BG$ to denote a classifying space of $G$. For $G=\Or(n)$ we have $BG=G_n$, where $G_n$ denotes the infinite Grassmannian. Let $\gamma^n$ denote the real $n$-plane vector bundle over $G_n$ and $\phi_n$ denote the standard representation of $\Or(n)$ on $\mathbb{R}^n$. In fact the vector bundle $\gamma^n$ is isomophic to the vector bundle  associated to $\phi_n$ over $G_n$. From \cite[Theorem $B.7$, page $381$]{spingeo} we obtain
\begin{equation}\label{cron}
H^{*}(\Or(n),\Z/2\Z)=\Z/2\Z[w_1,\ldots,w_n],
\end{equation}
where $w_i$ denotes the $i$-th Stiefel-Whitney class of the vector bundle $\gamma^n$ over $G_n$. In other words $w_i=w_i(\gamma^n)=w_i(\phi_n)$.
From the same reference we obtain 
$H^{*}(\SO(n),\Z/2\Z)=\Z/2\Z[w_2',\ldots,w_n']$, where $$w_j'\in H^j(\SO(n),\Z/2\Z)\quad \text{and}\quad w_j'=w_j(i^{*}(\gamma^n)).$$
Here $i:\mathrm{B}\SO(n)\to \mathrm{B}\Or(n)$ denotes the induced map from the inclusion of $\SO(n)$ into $\Or(n)$.

Let `$\det$' denote the determinant of $\phi_n$. Note that $w_1(\phi_n)=w_1(\det)$ is the only non-zero element of $H^1(G,\Z/2\Z)$. We write $e_{\mathrm{cup}}=w_1(\det)\cup w_1(\det)$.
From Equation \eqref{cron} it follows that
$$H^2(\Or(n),\Z/2\Z)=\langle w_2(\gamma^n),e_{\mathrm{cup}}\rangle\cong\Z/2\Z\oplus \Z/2\Z.$$ 
For any representation $\pi$ of $\Or(n)$ we have
$$w_2(\pi)=aw_2(\gamma^n)+be_{\mathrm{cup}},$$
where $a,b\in \Z/2\Z$. For the subgroups $C_2$ and $\SO(n)$ of $\Or(n)$ we obtain the restriction maps $i_1^{*}:H^2(\Or(n),\Z/2\Z)\to H^2(\SO(n),\Z/2\Z)$ and $i_2^{*}:H^2(\Or(n),\Z/2\Z)\to H^2(C_2,\Z/2\Z)$.

\begin{lemma}\label{bij}
	The map 
	$$i^*: H^2(\Or(n),\Z/2\Z)\to H^2(\SO(n),\Z/2\Z)\oplus  H^2(C_2,\Z/2\Z),$$
	given by $i^*(\alpha)=i_1^*(\alpha)\oplus i_2^*(\alpha)$, for $\alpha\in H^2(\Or(n),\Z/2\Z)$, is an isomorphism.
\end{lemma}

\begin{proof}
Since $i^{*}$ is a linear map between $2$-dimensional $\Z/2\Z$ vector spaces, it suffices to show that its rank is $2$. From \cite[page no. $328$]{guna} we obtain that $w_2+w_1\cup w_1\in H^2(\Or(n),\Z/2\Z)$ corresponds to the group extension $\Pin(n)$ of $\Or(n)$, where $w_i$ denotes the Stiefel-Whitney classes of $\gamma^n$ over $G_n$. The restriction map $\phi_n\mid_{\SO(n)}:\SO(n)\to \SO(n)$ is in fact the identity map on ${\SO(n)}$. If there exists a lift $l:\SO(n)\to \Spin(n)$, then it becomes a section of the covering map $\rho:\Spin(n)\to \SO(n)$, as $\rho\circ l$ is the identity map on ${\SO(n)}$.
This violates the fact that $\Spin(n)$ is a non-trivial double cover of $\SO(n)$. Therefore $\phi_n\mid_{\SO(n)}$ is aspinorial and hence $i_1^*(w_2(\phi_n))= w_2(\phi_n \mid_{\SO(n)}) \neq 0$.

Note that $H^*(C_2, \Z/2\Z) = \Z/2\Z [x]$.
Also observe that $\phi_n\mid_{C_2}$ is the identity map on ${C_2}$, 
 where
\[
C_2=
\begin{cases}
\{\pm I\},\,\,\text{for $n$ is odd},\\
\langle g_0 \rangle,\,\, \text{for $n$ is even}.
\end{cases}
\]
In any case we have
$$w_1(\phi_n\mid_{C_2})=\det\circ \phi_n\mid_{C_2}=x .$$
The representation $\phi_n\mid_{C_2}$ is spinorial if and only if $n\equiv 3 \pmod 4$. 
Thus in this case $i_2^*(w_2(\phi_n\mid_{C_2})) + x^2 = 0$ which implies $i_2^*(w_2) = x^2$, and $0$ otherwise.
Therefore we obtain
		\[
		i_2^{*}(w_2(\phi_n))=
		\begin{cases}
		1,\quad \text{when $n\equiv 3\pmod 4$}\\
		0, \quad \text{otherwise}.
		\end{cases}
		\]
Therefore $i^{*}(w_2(\phi_n))=(1,1)$ or $(1,0)$.
		 
Consider the representation $\psi =\det\oplus \det$. As $\det\mid_{\SO(n)}= 1$, we obtain $i_1^*(w_2(\psi))=0$. Note that $\psi(g_0)= \begin{pmatrix}
			-1 & 0\\
             0 & -1
		 \end{pmatrix}.$ So $w_1(\psi\mid_{C_2})=\det(\psi\mid_{C_2})=1$ and $\psi\mid_{C_2}$ is aspinorial. This gives $i^*(w_2(\psi))=(0,1)$. Thus $i^*$ has rank $2$, as required.


%
\end{proof}

\begin{theorem}\label{thm1}
	
Let $\pi_{\lambda}\otimes \rho$ be an irreducible representation of $\Or(2n+1)$, where $\rho$ is an irreducible representation of $C_2$ and $\pi_{\lambda} $ be the irreducible representation of $\SO(2n+1)$ with highest weight $\lambda$. Then 
$$w_2((\pi_{\lambda}\otimes \rho)_0)=\dfrac{(2n-1)}{n(2n+1)}\dim V^{\lambda}\cdot \chi_{\lambda}(C)w_2(\gamma^n)+\dfrac{m(m-1)}{2}e_{\mathrm{cup}},$$
where 
\[
m=
\begin{cases}
\dim \pi_{\lambda},\,\,\text{for}\,\,\rho=\sgn,\\
0,\,\,\text{otherwise}.
\end{cases}
\]
\end{theorem}

\begin{proof}
From Lemma \ref{bij} we obtain 
$$w_2(\pi_0)=aw_2(\gamma^n)+be_{\mathrm{cup}}.$$
We choose $a=\dfrac{(2n-1)}{n(2n+1)}\dim V^{\lambda}\cdot\chi_{\lambda}(C)$ by \cite[Remark $5$]{joshi} and $b=\dfrac{m(m-1)}{2}$. Note that $\pi_{\lambda}\otimes \rho\mid _{\SO_{2n+1}}$ is spinorial if and only if $a\equiv0\pmod2$. We use Lemma \ref{lem2}  for $b$ . Hence the result follows.
\end{proof}

\begin{theorem}\label{sw2}
Let $(\pi,W)$ be an irreducible representation of $\Or(2n)$, where $n\geq 4$. Let $R^{+}$ be the root system $C_{n-1}$. Then we have

$$w_2(\pi_0)=
\dfrac{2 \cdot (n-1)}{n\cdot(2n-1)}\cdot\dim V^{\lambda}\left(\chi_{\lambda}(C)+ \chi_{g_0\cdot\lambda}(C)\right)w_2(\gamma^n)
+\dfrac{m(m-1)}{2}e_{\mathrm{cup}},$$
when $\pi = \rho^{\lambda}$ is irreducible.
For $\pi=\pi^{\lambda, \pm}$, we have
$$
w_2(\pi_0)=\dfrac{2 \cdot (n-1)}{n\cdot(2n-1)}\dim V^{\lambda}\cdot \chi_{\lambda}(C)w_2(\gamma^n)+\dfrac{m(m-1)}{2}e_{\mathrm{cup}}.$$
 Here $m$ is as mentioned in Proposition \ref{eigendim} and $\chi_{ \pi^{\lambda , \pm}}(g_0)$ is as mentioned in Theorem \ref{chi}.
\end{theorem}

\begin{proof}
The proof follows by a similar argument as in Theorem \ref{thm1}. 
\end{proof}

\begin{remark}
We have $H^2(\SO(l),\Z/2\Z)=\langle w_2'\rangle$. Therefore for an orthogonal representation $\pi^{\lambda}$ of $\SO(n)$ by \cite[Remark $5$]{joshi} we obtain 
\[
w_2((\pi^{\lambda})_0)=
\begin{cases}
\dfrac{(2n-1)}{n(2n+1)}\dim V^{\lambda}\cdot \chi_{\lambda}(C)w'_2 ,\,\,\text{when $l=2n+1$},\\
\dfrac{2 \cdot (n-1)}{n\cdot(2n-1)}\dim V^{\lambda}\cdot \chi_{\lambda}(C)w'_2,\,\,\text{when $l=2n$}.
\end{cases}
\]	
\end{remark}
\subsection{Calculation of $w_1$}

Let $(\pi,W)$ be an irreducible representation of $\Or(n)$. From \cite{DPDR} we obtain $w_1(\pi_0)=\det(\pi_0)$. 

\begin{theorem}
For $\Or(2n+1)$, if $\pi=\pi^{\lambda}\otimes \rho$ as in Corollary \ref{thm1}, then 	
\[
w_1(\pi_0)=
\begin{cases}
\dim V^{\lambda}\cdot w_1(\gamma^n),\,\,\text{when}\,\,\rho=\sgn\\
0,\,\,\text{when}\,\,\rho=\mathbb{1}.
\end{cases}
\]
For $\Or(2n)$ 
$$w_1(\pi_0)=m\cdot w_1(\gamma^n),$$
where $m$ is as in Proposition \ref{eigendim}.
\end{theorem}	

\begin{proof}
Note that $\pi(\SO(n))\subset \mathrm{SL}(n)$. Therefore it is enough to determine $\det(\pi(g_0))$ (resp. $\det(\pi(-I))$) for $\Or(2n)$ (resp. $\Or(2n+1)$).
\end{proof}
\section{A Character formula }\label{chfor}

We begin this section with a detection result.

\begin{prop}\label{detection}
	Let $G_1$ and $G_2$ be two groups and $f:G_1 \to G_2$ be a morphism. If the map $f^*:H^i(G_2) \to H^i(G_1)$ is injective for $i\in\{1,2\}$, then any representation $\phi$ of $G_2$ is spinorial if and only if $\phi \circ f$ is spinorial. 
\end{prop}
\begin{proof}
	We have $w_i(\phi \circ f)=f^*(w_i(\phi))$, for $i\in\{1,2\}$. This gives
	\begin{equation}\label{sw1}
	w_2(\phi\circ f) + w_1(\phi \circ f)\cup  w_1(\phi \circ f)=f^*(w_2(\phi) + w_1(\phi)\cup  w_1(\phi)).
	\end{equation}
	
	From \cite[page no. 238]{guna} we obtain $\phi$ is spinorial if and only if $w_2(\phi)) + w_1(\phi)\cup w_1(\phi)=0$. If $\phi$ is spinorial then following equation \eqref{sw1} we conclude that $\phi\circ f$ is spinorial. On the other hand if $\phi\circ f$ is spinorial then we obtain $f^*(w_2(\phi) + w_1(\phi)\cup  w_1(\phi))=0$. Since $f^{*}$ is injective we have the result.
\end{proof}

Consider the subgroup  
$$D_i=\diag\{\underbrace{\pm1, \pm1, \ldots,\pm1}_{i \,\,\text{times}},1,\ldots 1\ldots, 1\}$$
of $\Or(n)$ consisting of the diagonal matrices for $1\leq i\leq n$ . We write  
$$d_i=\diag(\underbrace{-1, -1, \ldots,-1}_{i \,\,\text{times}},1,\ldots 1\ldots, 1).$$
The next result enables us to detect spinorial representations of $\Or(n)$ from their character values.
\begin{theorem}
	A representation $\pi$ of $\Or(n)$ is spinorial if and only if both the following conditions hold:
	\begin{enumerate}
		\item $\chi_{ \pi}(1) -\chi_{ \pi}(d_1) \equiv 0 \text{ or } 6 \pmod 8$,
		\item $\chi_{ \pi}(1) -\chi_{ \pi}(d_2) \equiv 0 \text{ or } 6 \pmod 8$.
	\end{enumerate}
\end{theorem}

\begin{proof}
	Let $W_i$ denote the permutation group of $D_i$. Note that $W_i=S_i$. From \cite[Theorem $2.2$, page $82$]{toda} we obtain an isomorphism $\alpha: H^{*}(\mathrm{O}(n))\to H^{*}(D_n)^{W_n}$. From \cite[Theorem $4.4$, page $69$]{adem} we have $H^{*}(D_i)=\ztwo[x_1,x_2,\ldots,x_i]$, the invariant subgroup $H^*(D_i)^{W_i}=\ztwo[x_1,x_2,\ldots,x_i]^{S_i}$ is the ring of symmetric polynomials in $i$ variables. We obtain an injection $f_i:H^i(D_n)^{W_n} \to H^i(D_i)^{W_i}$ by putting $x_j=0$ for $i+1\leq j\leq n$.  Therefore we obtain an injection 
	$$f_i\circ \alpha:H^i(\Or(n))\to H^i(D_i)^{W_i}.$$
	
	In particular $$ f_2 \circ \alpha : H^2(O(n)) \hookrightarrow H^2(D_2)$$ is an injection. From Proposition \ref{detection} we conclude that $\pi$ is spinorial if and only if $\pi\mid_{D_2}$ is spinorial. The non-trivial elements of the group $D_2$ are the involutions $d_1, d_2$ and $d_1d_2$. Therefore $D_2=\langle d_1,d_2\rangle\cong C_2\times C_2$. From Lemma \ref{lem1} we obtain that $\pi\mid_{C_2=\langle d_1\rangle}$ is spinorial if and only if the first condition mentioned in the theorem holds. Since $d_1$ and $d_1d_2$ are conjugate in $\Or(n)$, the same condition works for $\pi\mid_{C_2=\langle d_1d_2\rangle}$. Similarly the representation $\pi\mid_{C_2=\langle d_2\rangle}$ is spinorial if and only if the second condition holds.
\end{proof}
Using Theorem \ref{thm2} one obtains similar result for orthogonal representations of $\mathrm{GL_n}(\mathbb{R})$. We state it as the following corollary.
\begin{corollary}
	An orthogonal representation $\pi$ of $\mathrm{GL_n}(\mathbb{R})$ is spinorial if and only if both the following conditions hold:
	\begin{enumerate}
		\item $\chi_{ \pi}(1) -\chi_{ \pi}(d_1) \equiv 0 \text{ or } 6 \pmod 8$,
		\item $\chi_{ \pi}(1) -\chi_{ \pi}(d_2) \equiv 0 \text{ or } 6 \pmod 8$.	
	\end{enumerate}
\end{corollary}

\begin{remark}
One can use the injective map $f_2 \circ \alpha : H^2(O(n)) \hookrightarrow H^2(D_2)$ to calculate $w_2(\pi_0)$ for a representation $\pi$ of $\Or(n)$.
Consider the two $1$-dimensional representations of $C_2\times C_2$, namely
$\phi_2: (d_1,d_2)\to (-1,1)$ and $\phi_3: (d_1,d_2)\to (1,-1)$. We write $g_{d_i}=\frac{\chi_{\pi}(1)-\chi_{ \pi}(d_i)}{2}$ for $i\in \{1,2\}$ and $g_{d_1d_2}=\frac{\chi_{\pi}(1)-\chi_{ \pi}(d_1d_2)}{2}$. 
For a representation $(\pi,V)$ of $\Or(n)$ we have 
$$w_2(\pi_0)=\left[\frac{g_{d_2}}{2}\right]\alpha^2+\left[\frac{g_{d_2}}{2}\right]\beta^2+\left(\left[\frac{g_{d_1d_2}}{2}\right]+\left[\frac{g_{d_1}}{2}\right]+\left[\frac{g_{d_2}}{2}\right]\right)\alpha\beta,$$
where $\alpha=w_1(\phi_2)$, $\beta=w_1(\phi_3)$ and $[\cdot]$ denotes the greatest integer function. For details we refer to \cite{sujeet}.
\end{remark}

\section{Examples}\label{examples}

We work out the cases for the irreducible representations of $\Or(2)$ and $\Or(4)$ and $\Or(8)$. Note that
we cannot apply Theorem \ref{spin1} to determine the irreducible spinorial representations of the first two groups as $\SO(2)$ and $\SO(4)$ are not simple.

\subsection{Case of $\Or(2)$}

We have $\Or(2)=S^1\rtimes C_2$. The irreducible representations of $S^1$ are $\pi_n$ given by $\pi_n(z)=z^n$. We write $\rho_n=\mathrm{Ind}_{\SO(2)}^{\Or(2)}(\pi_n)$. From \cite[Section $5.1$]{repofO2} we obtain that for $n>0$, the irreducible representations are given by $\rho_n$. On the other hand $\rho_0=\mathbb{1}\oplus\det$.
We write $g_0$ to denote the element $\begin{pmatrix}
0 & 1\\
1 & 0
\end{pmatrix}$.
From Theorem \ref{chwendt} we obtain $\chi_{\rho_n}(g_0)=0$. Therefore $\rho_n(g_0)\sim\begin{pmatrix}
-1 & 0\\
0 & 1
\end{pmatrix}$. Hence by Lemma \ref{lem1}, the representations $\rho_n$ are aspinorial for $n>0$. The 
$\det$ representation is aspinorial by the same lemma.

\subsection{Case of $\Or(4)$} 

We know that $\Spin(4)\cong \SU_2\times \SU_2$, whose irreducible representations are of the form
$V_{a,b}=\mathrm{Sym}^aV_0\boxtimes \mathrm{Sym}^bV_0$, where $V_0$ denotes the standard representation of $\SU_2$. The representations which factor through $\SO(4)$ have the property $a\equiv b\pmod 2$. We write $g_0=\left(\sum_{i=1}^{2}e_{i,i}+e_{3,4}+e_{4,3}\right)
$.
From Proposition \ref{eigendim} we obtain 
$$m=\dim(V_{a,b})=(a+1)(b+1),$$
when $\rho^{(a,b)}$ is irreducible. So $\rho^{(a,b)}$ is spinorial if and only if both the following conditions hold:
\begin{enumerate}
\item 
$(a+1)(b+1)\equiv 0\,\,\text{or}\,\,3\pmod 4$
\item 
$\frac{1}{4}\left((a+1)\binom{b+2}{3}+(b+1)\binom{a+2}{3}\right)\equiv 0\pmod 2$.
\end{enumerate}
We have the second condition from \cite[Example $3$, page $21$]{joshi}.
Now we consider the case when $\rho^{a,b}$ is reducible. The root system of $SO(4)$ is $D_2\cong A_1\times A_1$. Therefore $D_2^{\tau}\cong A_1^{\tau}\times A_1^{\tau}\cong A_1\times A_1$. From \cite[Exercise $14.36$, page $210$]{fulton} we obtain the Killing form  $(.)$ for $\mf{so}(2n)$ as 
\begin{equation}\label{kill}
(\mu_1,\mu_2)=\frac{1}{2(2n-2)}\mu_1\cdot\mu_2,
\end{equation}
where $\mu_1\cdot \mu_2$ denotes the normal inner product. The positive roots are $S=\{e_1-e_2, e_1+e_2\}$. We calculate $$(e_1-e_2)^{\vee}=4 \cdot \frac{2(e_1-e_2)}{(e_1-e_2)\cdot(e_1-e_2)}=4(e_1-e_2).$$ 
Similarly we have $(e_1+e_2)^{\vee}=4(e_1+e_2)$.
The highest weight of $V_{a,b}$ is $(\frac{a+b}{2},\frac{b-a}{2})$. 
Now for calculating $\rho^{\tau}$ we take the normalized Killing form mentioned in section \ref{wendt}.
\begin{align*}
\rho^{\tau} &= \frac{1}{2}\sum_S \frac{2 \alpha}{(\alpha,\alpha)} \\
&= \frac{1}{2}(e_1+e_2+e_1-e_2)\\
&= e_1.
\end{align*}

From Theorem \ref{chi} we obtain 
$$\chi_{\pi^{\lambda,\pm}}(g_0)=\pm\dfrac{\prod_{\alpha\in S}\langle \alpha^{\vee}, \lambda+\rho^{\tau}\rangle}{\prod_{\alpha\in S}\langle \alpha^{\vee}, \rho^{\tau}\rangle}.$$
 Here $\langle\alpha^{\vee},x\rangle=(\frac{2\alpha}{(\alpha,\alpha)},x)$ where (,) is the killing form. We have $$\lambda+\rho^{\tau}=(\frac{a+b}{2}+1){e}_1+(\frac{b-a}{2}){e}_2$$.
So we calculate
\begin{align*}
\chi_{\pi^{\lambda,\pm}}(g_0)&=\pm \dfrac{( e_1-e_2,(\frac{a+b}{2}+1){e}_1+(\frac{b-a}{2}){e}_2 ) ( e_1+e_2,(\frac{a+b}{2}+1){e}_1+(\frac{b-a}{2}){e}_2 )}{( e_1-e_2,{e}_1) ( e_1+e_2,{e}_1)}\\
&=\pm \dfrac{(a+1)(b+1)}{1}\\
&=\pm (a+1)(b+1).
\end{align*}

From Proposition \ref{eigendim} we obtain 

\[
m=
\begin{cases}
0 \quad \text{when}\quad \pi=\pi^{\lambda,+}\\
(a+1)(b+1) \quad \text{when}\quad \pi=\pi^{\lambda,-}.
\end{cases}
\]

Therefore we conclude the representation $\pi^{\lambda,-}$ is spinorial if and only if both the following conditions hold:
\begin{enumerate}
\item 
$(a+1)(b+1)\equiv 0\,\,\text{or}\,\,3\pmod4$
\item 
$\frac{1}{4}\left((a+1)\binom{b+2}{3}+(b+1)\binom{a+2}{3}\right)\equiv 0\pmod 2$.
\end{enumerate} 

On the other hand the representation $\pi^{\lambda,+}$ is spinorial if and only if 
$\frac{1}{4}\left((a+1)\binom{b+2}{3}+(b+1)\binom{a+2}{3}\right)\equiv 0\pmod 2$.

\subsection{Case of $\Or(8)$} 
Consider the representations $V^{\lambda}$ of $\SO(8)$ with highest weight  $\lambda=(\lambda_1,\lambda_2,\lambda_3,0)$. In these cases $\mathrm{Ind}_{\SO(8)}^{\Or(8)}V^{\lambda}$ is reducible with irreducible components $\pi^{\lambda,\pm}$ such that $\pi^{\lambda,\pm}\mid_{\SO(8)}=V^{\lambda}$. From \cite[Table $1$, Section $11$]{joshi} we have the representation $V^{\lambda}$ is spinorial if and only if the integer 
\begin{equation}\label{SOspin}
\dfrac{2(n-1)}{n(2n-1)}\dim V^{\lambda}\cdot \chi_{\lambda}(C),
\end{equation}
is even. From \cite[Theorem $1.7$, page $242$]{BrokerDieck} we obtain the Weyl dimension formula
\begin{equation}\label{dimcal}
\dim V^{\lambda}=\prod_{\alpha\in \phi^{+}}\dfrac{(\lambda+\rho,\alpha)}{(\rho,\alpha)},
\end{equation}
where $(.)$ denotes the killing form as mentioned in \eqref{kill}.

As an example here we solve for the case  $$\lambda=(1,0,0,0).$$
The root system for $\SO(8)$ is $D_4$. The positive roots $\phi^{+}$ are $e_i-e_j$, $e_i+e_j$ for $1\leq i<j\leq 4$. The half sum of positive roots is $\rho=(3,2,1,0)$. , then using Equation \ref{dimcal} we calculate
\begin{align*}
\dim V^{\lambda}&=\dfrac{\prod_{1\leq i<j\leq 4}((4,2,1,0)\cdot (e_i-e_j))\prod_{1\leq i<j\leq 4}((4,2,1,0)\cdot (e_i+e_j))}{\prod_{1\leq i<j\leq 4}((3,2,1,0)\cdot (e_i-e_j))\prod_{1\leq i<j\leq 4}((3,2,1,0)\cdot (e_i+e_j))}\\
&=8.
\end{align*} 

From \cite[Section $2.3$]{joshi} we have $\chi_{\lambda}(C)=(\lambda,\lambda+2\rho)$. This gives 
\begin{align*}
\chi_{\lambda}(C)&=\frac{1}{12}\left((1,0,0,0)\cdot((1,0,0,0)+2(3,2,1,0)) \right)\\
&=7/12.
\end{align*} 
Putting $n=4$ in Equation \ref{SOspin} we obtain
\begin{equation}
\dfrac{2(n-1)}{n(2n-1)}\dim V^{\lambda}\cdot \chi_{\lambda}(C)=\frac{6}{28}\cdot 8\cdot\frac{7}{12}=1.
\end{equation}
Therefore $\pi^{(1,0,0,0),\pm}$ is aspinorial.

To find the Stiefel- Whitney class of this representation we need $m$ as in Proposition \ref{eigendim}. We take $\mu = \lambda \mid_{LS_0^*} = (1,0,0)$. The root system of $\SO(8)$ is $D_4$. From Table \ref{table} we have $D_4^{\tau}= B_3$. Let $S$ be the set of positive roots of $B_3$. Then $S = \{e_i \pm e_j \mid 1\leq i < j \leq 3\} \cup \{ e_i \mid 1 \leq i \leq 3\}$. Therefore we calculate 

\begin{align*}
\rho^{\tau} &= \frac{1}{2}\sum_S \frac{2 \alpha}{(\alpha,\alpha)} \\
    &= \sum_S \frac{\alpha}{(\alpha,\alpha)}\\
    &= (3,2,1).
\end{align*}

From Theorem \ref{chi} we obtain 
$$\chi_{\pi^{\lambda,\pm}}(g_0)=\pm\dfrac{\prod_{\alpha\in S}\langle \alpha^{\vee}, \lambda+\rho^{\tau}\rangle}{\prod_{\alpha\in S}\langle \alpha^{\vee}, \rho^{\tau}\rangle},$$
where $\alpha^{\vee}=\frac{2\alpha}{(\alpha,\alpha)}$. Putting the value of $\alpha^{\vee}$ we calculate
\begin{align*}
\chi_{\pi^{\lambda,\pm}}(g_0)&=\pm \dfrac{\prod_{\alpha\in S}( \alpha, \lambda+\rho^{\tau})}{\prod_{\alpha\in S}( \alpha, \rho^{\tau})}\\
&=\pm \dfrac{\prod_{\alpha\in S}( \alpha, (4,2,1))}{\prod_{\alpha\in S}( \alpha, (3,2,1))}\\
&=\pm 6.
\end{align*}
From Theorem \ref{eigendim} we have $m=7$ for $\pi^{\lambda,+}$ and $m=1$ for $\pi^{\lambda,-}$.
Putting these values in Theorem \ref{sw2} we obtain
\[
w_2(\pi)=
\begin{cases}
w_2(\gamma),\,\,\text{for}\,\,\pi=\pi^{\lambda,+}\\
w_2(\gamma)+e_{\mathrm{cup}},\,\,\text{for}\,\,\pi=\pi^{\lambda,-}.
\end{cases}
\]
For the notations $\gamma$ and $e_{\mathrm{cup}}$ see Section \ref{sw}.
\bibliographystyle{alpha}
\bibliography{mybib} 

\end{document}